\documentclass[reqno,oneside]{amsproc}
\DeclareFontEncoding{OT2}{}{}
\DeclareFontSubstitution{OT2}{cmr}{m}{n}
\DeclareFontFamily{OT2}{cmr}{\hyphenchar\font45 }
\DeclareFontShape{OT2}{cmr}{m}{n}{
<5><6><7><8><9>gen*wncyr
<10><10.95><12><14.4><17.28><20.74><24.88>wncyr10}{}
\DeclareFontShape{OT2}{cmr}{b}{n}{
<5><6><7><8><9>gen*wncyb
<10><10.95><12><14.4><17.28><20.74><24.88>wncyb10}{}
\DeclareMathAlphabet{\mathcyr}{OT2}{cmr}{m}{n}
\DeclareMathAlphabet{\mathcyb}{OT2}{cmr}{b}{n}
\SetMathAlphabet{\mathcyr}{bold}{OT2}{cmr}{b}{n}
\usepackage{geometry}
\usepackage{fancybox,ascmac,comment}
\usepackage{amssymb}
\usepackage{amsfonts}
\usepackage{amsmath}
\usepackage{amsthm}
\usepackage{mathtools}
\usepackage{enumerate}

\usepackage{mleftright}
\mleftright
\usepackage[all]{xy}
\mathtoolsset{showonlyrefs=true}
\numberwithin{equation}{section}
\makeatletter
\newcommand{\shortmathcal}[1]{\@tfor\ch:=#1\do{
\expandafter\edef\csname c\ch\endcsname{\noexpand\mathcal{\ch}}
}}
\newcommand{\shortmathbb}[1]{\@tfor\ch:=#1\do{
\expandafter\edef\csname bb\ch\endcsname{\noexpand\mathbb{\ch}}
}}
\newcommand{\shortmathbf}[1]{\@tfor\ch:=#1\do{
\expandafter\edef\csname b\ch\endcsname{\noexpand\mathbf{\ch}}
}}
\newcommand{\shortboldsymbol}[1]{\@tfor\ch:=#1\do{
\expandafter\edef\csname bs\ch\endcsname{\noexpand\boldsymbol{\ch}}
}}
\newcommand{\shortmathfrak}[1]{\@tfor\ch:=#1\do{
\expandafter\edef\csname f\ch\endcsname{\noexpand\mathfrak{\ch}}}}
\newcommand{\shortmathscr}[1]{\@tfor\ch:=#1\do{
\expandafter\edef\csname s\ch\endcsname{\noexpand\mathscr{\ch}}}}
\newcommand{\shortmathrm}[1]{\@tfor\ch:=#1\do{
\expandafter\edef\csname r\ch\endcsname{\noexpand\mathrm{\ch}}
}}
\makeatother
\shortmathcal{ABCDEFGHIJKLMNOPQRSTUVWXYZ}
\shortmathbb{ABCDEFGHIJKLMNOPQRSTUVWXYZ}
\shortmathbf{ABCDEFGHIJKLMNOPQRSTUVWXYZabcdefghijklmnopqrstuvwxyz}
\shortboldsymbol{ABCDEFGHIJKLMNOPQRSTUVWXYZabcdefghijklmnopqrstuvwxyz}
\shortmathfrak{ABCDEFGHIJKLMNOPQRSTUVWXYZabcdefghjklmnopqrstuvwxyz}
\shortmathscr{ABCDEFGHIJKLMNOPQRSTUVWXYZ}
\shortmathrm{ABCDEFGHIJKLMNOPQRSTUVWXYZabcdefghijklmnopqrstuvwxyz}
\DeclareFontEncoding{OT2}{}{}
\DeclareFontSubstitution{OT2}{cmr}{m}{n}
\DeclareFontFamily{OT2}{cmr}{\hyphenchar\font45 }
\DeclareFontShape{OT2}{cmr}{m}{n}{
<5><6><7><8><9>gen*wncyr
<10><10.95><12><14.4><17.28><20.74><24.88>wncyr10}{}
\DeclareFontShape{OT2}{cmr}{b}{n}{
<5><6><7><8><9>gen*wncyb
<10><10.95><12><14.4><17.28><20.74><24.88>wncyb10}{}
\DeclareMathAlphabet{\mathcyr}{OT2}{cmr}{m}{n}
\DeclareMathAlphabet{\mathcyb}{OT2}{cmr}{b}{n}
\SetMathAlphabet{\mathcyr}{bold}{OT2}{cmr}{b}{n}
\newcommand{\sh}{\mathbin{\mathcyr{sh}}}
\newcommand{\emp}{\varnothing}

\newcommand{\hcA}{\widehat{\mathcal{A}}}
\newcommand{\hcS}{\widehat{\mathcal{S}}}

\newcommand{\hcRS}{\widehat{\mathcal{RS}}}

\newcommand{\jump}[1]{\ensuremath{[\![#1]\!]}}

\newcommand{\dch}{\mathrm{dch}}

\newcommand{\dep}{\mathrm{dep}}
\newcommand{\wt}{\mathrm{wt}}
\newcommand{\vv}[2]{\begin{pmatrix}#1\\ #2\end{pmatrix}}

\newcommand{\Li}{\mathrm{Li}}
\renewcommand{\Re}{\mathrm{Re}}

\theoremstyle{definition}
\newtheorem{theorem}{Theorem}[section]
\newtheorem{proposition}[theorem]{Proposition}
\newtheorem{lemma}[theorem]{Lemma}
\newtheorem{corollary}[theorem]{Corollary}

\newtheorem{definition}[theorem]{Definition}
\theoremstyle{remark}
\newtheorem{remark}[theorem]{Remark}

\allowdisplaybreaks
\title{Duality on symmetric multiple polylogarithms}
\author{Hanamichi Kawamura}
\address[Hanamichi Kawamura]{Department of Mathematics, Faculty of Science Division I, Tokyo University of Science, 1-3 Kagurazaka, Shinjuku-ku, Tokyo, 162-8601, Japan}
\email{1121026@ed.tus.ac.jp}
\allowdisplaybreaks
\begin{document}
\begin{abstract}
    There are three kinds of multiple polylogarithms; complex, finite and symmetric.
    The dualities for the complex and finite cases are known.
    In this paper, we present proofs of them via iterated integrals and its symmetric counterpart by a similar method.
\end{abstract}
\maketitle
\section{Introduction}\label{sec:intro}
\subsection{Notations and definitions}
Put $\cI\coloneqq\{\emp\}\sqcup\bigsqcup_{r=1}^{\infty}(\bbZ_{\ge 1})^{r}$ and call its element an \emph{index}.
Define $\dep\colon\cI\to\bbZ_{\ge 0}$ (resp.~$\wt\colon\cI\to\bbZ_{\ge 0}$) to be a map giving the number (resp.~sum) of components.
Put $\dep(\emp)=\wt(\emp)=0$.
We say that an index $\bk$ is \emph{admissible} if $\bk=\emp$ or its last component exceeds $1$.
We also define its variant with variables:
\[\cV\coloneqq\left\{\vv{\emp}{\emp}\right\}\cup\left\{\vv{z_{1},\ldots,z_{r}}{\bk}\left|\begin{array}{c}\bk\in\cI,\\r=\dep(\bk)\ge 1,\\z_{1},\ldots,z_{r}\in\bbC,\\|z_{1}|,\ldots,|z_{r}|\le 1,\\ (k_{r},z_{r})\neq 1\end{array}\right.\right\}.\]
Finally, we denote by $\{X\}^{n}$ the tuple $(\underbrace{X,\ldots,X}_{n})$.
\subsection{Complex MPLs}
The \emph{multiple polylogarithm} (\emph{MPL}) is the map $\Li^{\sh}\colon\cV\to\bbC$ defined as
\[\Li^{\sh}\vv{z_{1},\ldots,z_{r}}{k_{1},\ldots,k_{r}}\coloneqq\sum_{0=n_{0}<n_{1}<\cdots<n_{r}}\frac{z_{1}^{n_{1}-n_{0}}\cdots z_{r}^{n_{r}-n_{r-1}}}{n_{1}^{k_{1}}\cdots n_{r}^{k_{r}}},\]
for $k_{1},\ldots,k_{r}\in\bbZ_{\ge 1}$ and $z_{1},\ldots,z_{r}\in\bbD\coloneqq\{z\mid |z|\le 1\}\subseteq\bbC$.
Put $\Li^{\sh}\vv{\emp}{\emp}\coloneqq 1$ for the empty tuple $\emp$.
This is a generalization of the \emph{multiple zeta value} (\emph{MZV}):
\[\zeta(k_{1},\ldots,k_{r})\coloneqq\sum_{0<n_{1}<\cdots<n_{r}}\frac{1}{n_{1}^{k_{1}}\cdots n_{r}^{k_{r}}}=\Li^{\sh}\vv{1,\ldots,1}{k_{1},\ldots,k_{r}}\qquad (k_{1},\ldots,k_{r-1}\ge 1,~k_{r}\ge 2).\]
An index $\bk=(k_{1},\ldots,k_{r})\in(\bbZ_{\ge 1})^{r}$ satisfying the above condition (i.e., $r=0$ or the last component exceeds $1$) is called an \emph{admissible index}.
It is known that MZVs satisfy the \emph{duality}
\begin{equation}\label{eq:mzv_duality}
    \zeta(\bk)=\zeta(\bk^{\dagger}),
\end{equation}
where the dual $\bk^{\dagger}$ of an admissible index $\bk$ is defined as
\[\left(\{1\}^{a_{1}-1},b_{1}+1,\ldots,\{1\}^{a_{h}-1},b_{h}+1\right)^{\dagger}=\left(\{1\}^{b_{h}-1},a_{h}+1,\ldots,\{1\}^{b_{1}-1},a_{1}+1\right),\]
for $h,a_{1},\ldots,a_{h},b_{1},\ldots,b_{h}\in\bbZ_{\ge 1}$.\\

As a generalization of \eqref{eq:mzv_duality} for MPLs, the following theorem holds.
\begin{theorem}[{\cite[\S6.1]{bbbl01}}, {\cite[Theorem 3.4]{kms22}}]\label{thm:mpl_duality}
    Let $d,a_{1},\ldots,a_{d-1},b_{1},\ldots,b_{d-1}$ be positive integers, $\bl_{1},\ldots,\bl_{d}$ admissible indices and $w_{1},\ldots,w_{d-1}$ complex numbers.
    Assume
    \begin{enumerate}
        \item $\Re(w_{i})\le 1/2$ and $|w_{i}|\le 1$ for $1\le i<d$,
        \item $\Re(w_{1})\neq 1/2$ if $\bl_{1}=\emp$ and $a_{1}=1$,
        \item $|w_{d-1}|\neq 1$ if $\bl_{d}=\emp$ and $b_{d-1}=1$.
    \end{enumerate}
    Put $r_{i}\coloneqq\dep(\bl_{i})$, $s_{i}\coloneqq\dep((\bl_{i})^{\dagger})$,
    \[\vv{\bz}{\bk}=\vv{\{1\}^{r_{1}+a_{1}-1},w_{1},\ldots,\{1\}^{r_{d-1}+a_{d-1}-1},w_{d-1},\{1\}^{r_{d}}}{\bl_{1},\{1\}^{a_{1}-1},b_{1},\ldots,\bl_{d-1},\{1\}^{a_{d-1}-1},b_{d-1},\bl_{d}}\]
    and
    \[\vv{\bz}{\bk}^{\dagger}=\vv{\{1\}^{s_{d}+b_{d-1}-1},\frac{w_{d-1}}{w_{d-1}-1},\ldots,\{1\}^{s_{2}+b_{1}-1},\frac{w_{1}}{w_{1}-1},\{1\}^{s_{1}}}{(\bl_{d})^{\dagger},\{1\}^{b_{d-1}-1},a_{d-1},\ldots,(\bl_{2})^{\dagger},\{1\}^{b_{1}-1},a_{1},(\bl_{1})^{\dagger}}.\]
    Then we have
    \[\Li^{\sh}\vv{\bz}{\bk}=(-1)^{d-1}\Li^{\sh}\left(\vv{\bz}{\bk}^{\dagger}\right).\]
\end{theorem}
Note that \eqref{eq:mzv_duality} is recovered by $d=1$ in Theorem \ref{thm:mpl_duality}.
\subsection{Finite MPLs}
For $N\in\bbZ_{\ge 1}$, define the truncated series of the MPL
\[\Li_{<N}^{\sh}\vv{z_{1},\ldots,z_{r}}{k_{1},\ldots,k_{r}}\coloneqq\sum_{0=n_{0}<n_{1}<\cdots<n_{r}<N}\frac{z_{1}^{n_{1}-n_{0}}\cdots z_{r}^{n_{r}-n_{r-1}}}{n_{1}^{k_{1}}\cdots n_{r}^{k_{r}}}\]
and its \emph{star} variant
\[\Li_{<N}^{\sh,\star}\vv{z_{1},\ldots,z_{r}}{k_{1},\ldots,k_{r}}\coloneqq\sum_{0=n_{0}<n_{1}\le\cdots\le n_{r}<N}\frac{z_{1}^{n_{1}-n_{0}}\cdots z_{r}^{n_{r}-n_{r-1}}}{n_{1}^{k_{1}}\cdots n_{r}^{k_{r}}},\]
as elements of $\bbQ[z_{1},\ldots,z_{r}]$.
Note that the star value is written as a sum of non-star values:
\begin{equation}\label{eq:star_nonstar}
    \Li_{<N}^{\sh,\star}\vv{z_{1},\ldots,z_{r}}{k_{1},\ldots,k_{r}}=\sum_{d=0}^{r-1}\sum_{0=i_{0}<\cdots<i_{d+1}=r}\Li^{\sh}_{<N}\vv{z_{i_{0}+1},\ldots,z_{i_{d}+1}}{k_{i_{0}+1}+\ldots+k_{i_{1}},\ldots,k_{i_{d}+1}+\cdots+k_{i_{d+1}}}.
\end{equation}
When we consider $z_{1},\ldots,z_{r}$ as complex numbers with suitable convergence conditions, we have
\[\lim_{N\to\infty}\Li_{<N}^{\sh}\vv{z_{1},\ldots,z_{r}}{k_{1},\ldots,k_{r}}=\Li^{\sh}\vv{z_{1},\ldots,z_{r}}{k_{1},\ldots,k_{r}}.\] 
Mimicking this, we also define the star variant of the MPL as
\[\Li^{\sh,\star}\vv{z_{1},\ldots,z_{r}}{k_{1},\ldots,k_{r}}\coloneqq\lim_{N\to\infty}\Li_{<N}^{\sh,\star}\vv{z_{1},\ldots,z_{r}}{k_{1},\ldots,k_{r}}.\]
\begin{remark}\label{rem:generating_function}
    The generating function of $\Li_{<N}^{\sh,\bullet}$ about $N$ is explained in terms of $\Li^{\sh,\bullet}$.
    Indeed, we have
    \begin{align}
        \sum_{N=1}^{\infty}\Li_{<N}^{\sh}\vv{z_{1},\ldots,z_{r}}{k_{1},\ldots,k_{r}}X^{N}
        &=\sum_{\substack{0<n_{1}<\cdots<n_{r}\\ N\ge 1}}\frac{z_{1}^{n_{1}-n_{0}}\cdots z_{r}^{n_{r}-n_{r-1}}}{n_{1}^{k_{1}}\cdots n_{r}^{k_{r}}}X^{n_{r}+N}\\
        &=\frac{X}{1-X}\sum_{0<n_{1}<\cdots<n_{r}}\frac{(Xz_{1})^{n_{1}-n_{0}}\cdots (Xz_{r})^{n_{r}-n_{r-1}}}{n_{1}^{k_{1}}\cdots n_{r}^{k_{r}}}\\
        &=\frac{X}{1-X}\Li^{\sh}\vv{Xz_{1},\ldots,Xz_{r}}{k_{1},\ldots,k_{r}},
    \end{align}
    and similarly
    \[\sum_{N=1}^{\infty}\Li_{<N}^{\sh,\star}\vv{z_{1},\ldots,z_{r}}{k_{1},\ldots,k_{r}}X^{N}=\frac{X}{1-X}\Li^{\sh,\star}\vv{Xz_{1},\ldots,Xz_{r}}{k_{1},\ldots,k_{r}}\]
    holds.
\end{remark}
Let us construct an algebra where we consider $\Li_{<p}^{\sh}$ and $\Li_{<p}^{\sh,\star}$ for a prime number $p$.
For a positive integer $n$, a commutative ring $R$, we put
\[\cA_{n,R}\coloneqq\left.\left(\prod_{p\in\cP}R/p^{n}R\right)\right/\left(\bigoplus_{p\in\cP}R/p^{n}R\right),\]
where $\cP$ is the set of all prime numbers.
Since $\{\cA_{n,R}\}_{n}$ forms a projective system of a (discrete) topological ring together with natural transitions, we can take the projective limit $\hcA_{R}\coloneqq\varprojlim_{n}\cA_{n,R}$.
Note that there is a natural ring homomorphism
\[\pi_{R}\colon\prod_{p\in\cP}\varprojlim_{n}R/p^{n}R\to\hcA_{R},\]
obtained in \cite[Lemma 2.3]{seki19}.
\begin{definition}
    For an index $\bk=(k_{1},\ldots,k_{r})$, we define the \emph{$\bsp$-adic finite multiple polylogarithm} ($\bsp$-adic FMPL) by
    \[\pounds_{\hcA}^{\sh,\bullet}\vv{z_{1},\ldots,z_{r}}{\bk}\coloneqq\pi_{\bbZ[z_{1},\ldots,z_{r}]}\left(\left(\Li_{<p}^{\bullet}\vv{z_{1},\ldots,z_{r}}{\bk}\right)_{p\in\cP}\right)\qquad (\bullet\in\{\emp,\star\}),\]
    as an element of $\hcA_{\bbZ[z_{1},\ldots,z_{r}]}$.
\end{definition}
The following theorem seems to be a finite analog of Theorem \ref{thm:mpl_duality}.
\begin{theorem}[{\cite[Theorem 3.4]{seki19}}]\label{thm:fmpl_duality}
    Put
    \[\cL^{\star}_{\hcA}\vv{z_{1},\ldots,z_{r}}{\bk}\coloneqq\sum_{n=0}^{\infty}\left(\pounds_{\hcA}^{\sh,\star}\vv{z_{1},\ldots,z_{r},\{1\}^{n}}{\bk,\{1\}^{n}}-\frac{1}{2}\pounds_{\hcA}^{\sh,\star}\vv{1,z_{2},\ldots,z_{r},\{1\}^{n}}{\bk,\{1\}^{n}}\right)\bsp^{n}.\]
    Then, for indices $\bk_{1},\ldots,\bk_{d}$ ($d\ge 1$), we have
    \[\cL^{\star}_{\hcA}\vv{z_{1},\{1\}^{\dep(\bk_{1})-1},\ldots,z_{d},\{1\}^{\dep(\bk_{d})-1}}{\bk_{1},\ldots,\bk_{d}}=\cL^{\star}_{\hcA}\vv{1-z_{1},\{1\}^{\dep((\bk_{1})^{\vee})-1},\ldots,1-z_{d},\{1\}^{\dep((\bk_{d})^{\vee})-1}}{(\bk_{1})^{\vee},\ldots,(\bk_{d})^{\vee}},\]
    in $\hcA_{\bbZ[z_{1},\ldots,z_{d}]}$.
    Here, for an index $\bk=(k_{1},\ldots,k_{r})$, we denote by $\bk^{\vee}$ the index $(l_{1},\ldots,l_{s})$ satisfying
    \[(l_{s},\ldots,l_{2},l_{1}+1)=(k_{1},\ldots,k_{r}+1)^{\dagger}.\]
\end{theorem}
The natural surjection $\bbZ[z_{1},\ldots,z_{r}]\to\bbZ$; $z_{i}\mapsto 1$ induces the map $\hcA_{\bbZ[z_{1},\ldots,z_{r}]}\to\hcA\coloneqq\hcA_{\bbZ}$ and this construction is compatible with $\pi_{R}$'s.
Through this assignment, we define the \emph{$\bsp$-adic finite multiple zeta values} (\emph{$\bsp$-adic FMZV}) as follows:
\[\zeta_{\hcA}^{\bullet}(\bk)\coloneqq\pi_{\bbZ}\left(\left(\Li^{\sh,\bullet}_{<p}\vv{\{1\}^{\dep(\bk)}}{\bk}\right)_{p\in\cP}\right)\qquad (\bullet\in\{\emp,\star\}).\]
This value satisfies the \emph{$\bsp$-adic duality}
\begin{equation}\label{eq:fmzv_duality}
    \sum_{n=0}^{\infty}\zeta_{\hcA}^{\star}(\bk,\{1\}^{n})\bsp^{n}=-\sum_{n=0}^{\infty}\zeta_{\hcA}^{\star}(\bk^{\vee},\{1\}^{n})\bsp^{n},
\end{equation}
which immediately follows from Theorem \ref{thm:fmpl_duality}.
\subsection{Symmetric MPLs}
The refined Kaneko--Zagier conjecture (for example, see \cite[Conjecture 4.3]{osy21}) asserts that there is a strong relationship between the value $\zeta_{\hcA}(\bk)$ appearing in the previous subsection and the \emph{$t$-adic symmetric multiple zeta value} (\emph{$t$-adic SMZV}) defined below.
Put
\begin{multline}
    \zeta_{\hcS}^{\sh}(k_{1},\ldots,k_{r})\coloneqq\sum_{i=0}^{r}(-1)^{k_{i+1}+\cdots+k_{r}}\zeta^{\sh}(k_{1},\ldots,k_{i})\\
    \cdot\sum_{n_{i+1},\ldots,n_{r}\ge 0}\left(\prod_{j=i+1}^{r}\binom{k_{j}+n_{j}-1}{n_{j}}\right)\zeta^{\sh}(k_{r}+n_{r},\ldots,k_{i+1}+n_{i+1})t^{n_{i+1}+\cdots+n_{r}}\in\bbR\jump{t},
\end{multline}
where $\zeta^{\sh}(\bl)$ is defined as follows: there uniquely exists $P(x)\in\bbR[x]$ such that
\[\Li^{\sh}\vv{1-z,\ldots,1-z}{\bl}=P(\log z)+O(z(\log z)^{J}),\qquad\text{as }z\to +0,\]
for some $J>0$, and we put $\zeta^{\sh}(\bl)\coloneqq P(0)$.
Actually $\zeta_{\hcS}^{\sh}(\bk)$ is in the power series ring $\cZ\jump{t}$ over $\cZ$ ($=$ the $\bbQ$-algebra generated by all MZVs).
This value modulo $\zeta(2)$, i.e., $\zeta_{\hcS}(\bk)\coloneqq\zeta_{\hcS}^{\sh}(\bk)\in\cZ/\zeta(2)\cZ$, is called the $t$-adic SMZV.
Immitating \eqref{eq:star_nonstar}, we put
\[\zeta^{\star}_{\hcS}(k_{1},\ldots,k_{r})\coloneqq\sum_{d=0}^{r-1}\sum_{0=i_{0}<\cdots<i_{d+1}=r}\zeta_{\hcS}(k_{i_{0}+1}+\cdots+k_{i_{1}},\ldots,k_{i_{d}+1}+\cdots+k_{i_{d+1}}).\]
One of the evidences of the (refined) Kaneko--Zagier conjecture is the $t$-adic duality:
\begin{theorem}[{\cite[Corollary 5.5]{tt23}}]\label{thm:smzv_duality}
    For an index $\bk$, we have
    \[\sum_{n=0}^{\infty}\zeta_{\hcS}^{\star}(\bk,\{1\}^{n})t^{n}=-\sum_{n=0}^{\infty}\zeta_{\hcS}^{\star}(\bk^{\vee},\{1\}^{n})t^{n}.\]
\end{theorem}
As far as the author knows, any correspondence generalizing the (refined) Kaneko--Zagier conjecture for (finite and symmetric) multiple polylogarithms is not explicitly conjectured. However, we can define the $t$-adic symmetric multiple polylogarithm
\begin{multline}
    \pounds_{\hcS,\alpha}^{\sh}\vv{z_{1},\ldots,z_{r}}{k_{1},\ldots,k_{r}}\coloneqq\sum_{\substack{0\le i\le r\\ z_{i+1}\neq 0}}(-1)^{k_{i+1}+\cdots+k_{r}}z_{i+1}^{\alpha}\Li^{\sh}\vv{z_{1}/z_{i+1},\ldots,z_{i}/z_{i+1}}{k_{1},\ldots,k_{i}}\\
    \cdot\sum_{n_{i+1},\ldots,n_{r}\ge 0}\left(\prod_{j=i+1}^{r}\binom{k_{j}+n_{j}-1}{n_{j}}\right)\Li^{\sh}\vv{z_{r+1}/z_{i+1},\ldots,z_{i+2}/z_{i+1}}{k_{r}+n_{r},\ldots,k_{i+1}+n_{i+1}}t^{n_{i+1}+\cdots+n_{r}},
\end{multline}
where $\alpha\in\bbZ$ and $z_{r+1}\coloneqq 1$.
The variables $z_{1},\ldots,z_{r}$ are taken from a set $Q$ satisfying some technical conditions (see the beginning of \S\ref{sec:main}).
This quantity $\pounds_{\hcS,\alpha}^{\sh}$ lives in the quotient $(\cZ_{Q}/2\pi i\cZ_{Q})\jump{t}$ of the power series ring over a certain ring $\cZ_{Q}/2\pi i\cZ_{Q}$ associated with $Q$.
\subsection{Purpose of this paper}
This paper is devoted to give a symmetric counterpart of Theorem \ref{thm:fmpl_duality}.
Define $\pounds_{\hcS,\alpha}^{\sh,\star}$ by mimicking \eqref{eq:star_nonstar} and put
\[\cL_{\hcS,\alpha}^{\star}\vv{z_{1},\ldots,z_{r}}{k_{1},\ldots,k_{r}}\coloneqq\sum_{n=0}^{\infty}\left(\pounds_{\hcS,\alpha}^{\sh,\star}\vv{z_{1},\ldots,z_{r},\{1\}^{n}}{\bk,\{1\}^{n}}-\frac{1}{2}\pounds_{\hcS,\alpha}^{\sh,\star}\vv{1,z_{2},\ldots,z_{r},\{1\}^{n}}{\bk,\{1\}^{n}}\right)t^{n}\]
for an integer $\alpha$, an index $\bk\neq\emp$ and complex numbers $z_{1},\ldots,z_{r}$ ($r=\dep(\bk)$) satisfying some technical conditions ($=$ elements of $Q$). 
\begin{theorem}\label{thm:main}
    Let $z_{1},\ldots,z_{d}$ ($d\ge 1$) be elements of $Q$, $\alpha$ an integer and $\bk_{1},\ldots,\bk_{d}$ indices.
    Then we have
    \[\cL^{\star}_{\hcS,\alpha}\vv{z_{1},\{1\}^{\dep(\bk_{1})-1},\ldots,z_{d},\{1\}^{\dep(\bk_{d})-1}}{\bk_{1},\ldots,\bk_{d}}=\cL^{\star}_{\hcS,\alpha}\vv{1-z_{1},\{1\}^{\dep((\bk_{1})^{\vee})-1},\ldots,1-z_{d},\{1\}^{\dep((\bk_{d})^{\vee})-1}}{(\bk_{1})^{\vee},\ldots,(\bk_{d})^{\vee}}.\]
\end{theorem}
Putting $Q=\{0,1\}$ and $d=1$ in the above equality, we get Theorem \ref{thm:smzv_duality}.
We prove Theorem \ref{thm:main} in Section \ref{sec:main}.
The next section is devoted to prove Theorems \ref{thm:mpl_duality} and \ref{thm:fmpl_duality} by the nature of iterated integrals.
\section*{Acknowledgements}
The author would like to thank Prof.~Minoru Hirose for kindly giving him the unpublished manuscript.
He is also grateful to Prof.~Sho Yoshikawa for careful reading of the manuscript.
\section{Iterated integrals and dualities of complex and finite MPLs}
\subsection{Review on iterated integrals}
For a commutaive ring $R$ and a set $S$, denote by $R\langle S\rangle$ the free $R$-algebra generated by $S$.
We say that $\gamma$ is a \emph{path} if $\gamma$ is a continuous piecewise smooth map $\gamma\colon[0,1]\to\bbP^{1}(\bbC)$.
The composition (resp.~reverse) of such paths is denoted by $(\gamma_{1},\gamma_{2})\mapsto\gamma_{1}\cdot\gamma_{2}$ (order in which they pass) (resp.~$\gamma\mapsto\gamma^{-1}$).
These operations give rise to that of their homotopy classes.
\begin{definition}
    Let $\gamma\colon[0,1]\to\bbP^{1}(\bbC)$ be a path and $a_{1},\ldots,a_{k}\in\bbC\setminus\gamma((0,1))$ such that $a_{0}\coloneqq\gamma(0)\neq a_{1}$ and $a_{k+1}\coloneqq\gamma(1)\neq a_{k}$.
    Then we define
    \[I_{\gamma}(a_{0};a_{1},\ldots,a_{k};a_{k+1})\coloneqq\int_{0<t_{1}<\cdots<t_{k}<1}\prod_{i=1}^{k}\frac{d\gamma(t_{i})}{\gamma(t_{i})-a_{i}}.\]
\end{definition}
Such iterated integrals are invariant under a homotopic deformation of paths, that is, $I_{\gamma}$ depends only on the homotopy class $\gamma$ represents.\\

Moreover, let us give the definition of regularized iterated integrals.
\begin{definition}[Regularized iterated integrals]
Take a sufficiently small $1/2>z>0$ and define a ``shortened'' path $\gamma_{z,1-z}$ by $t\mapsto\gamma((1-z)t+zt)$.
Then, for $a_{1},\ldots,a_{k}$ (possibly $a_{1}=a_{0}$ or $a_{k}=a_{k+1}$), there uniquely exists $P(x)\in\bbC[x]$ such that 
\[I_{\gamma_{z,1-z}}(a_{0};a_{1},\ldots,a_{k};a_{k+1})=P(\log z)+O(z(\log z)^{J}),\qquad \text{as }z\to+0.\]
We define $I_{\gamma}(e_{a_{1}}\cdots e_{a_{k}})\coloneqq P(0)$ and $\bbQ$-linearly extend it to $I_{\gamma}\colon\fH_{\gamma}\coloneqq\bbQ\langle e_{z}\mid z\in\bbC\setminus\gamma((0,1))\rangle\to\bbC$.
\end{definition}
We usually identify with the formal symbol $e_{a}$ with the $1$-form $(t-a)^{-1}dt$.
Note that
\begin{equation}\label{eq:transformation_dual}
    e_{a}(b+ct)=c\frac{dt}{(b+ct)-a}=e_{(a-b)/c}(t),
\end{equation}
and
\begin{equation}\label{eq:transformation_inverse}
    e_{a}(t^{-1})=-\frac{dt}{t^{2}(t^{-1}-a)}=\left(\frac{1}{(t-a^{-1})}-\frac{1}{t}\right)\,dt=(e_{1/a}-e_{0})(t)
\end{equation}
hold for $a,b\in\bbC$.
For (regularized) iterated integrals, the following formulas hold:
\begin{proposition}\label{prop:basic_ii}
    Let $\gamma$ and $\gamma'$ be composable paths.
    Then, for $a_{1},\ldots,a_{k}\in\bbC\setminus\gamma((0,1))$ and $w,w'\in\fH_{\gamma}$ such that the following iterated integrals exist, we have 
    \begin{enumerate}
        \item(Path composition formula) $\displaystyle I_{\gamma\cdot\gamma'}(e_{a_{1}}\cdots e_{a_{k}})=\sum_{i=0}^{k}I_{\gamma}(e_{a_{1}}\cdots e_{a_{i}})I_{\gamma'}(e_{a_{i+1}}\cdots e_{a_{k}})$.
        \item(Reversal formula) $I_{\gamma}(e_{a_{1}}\cdots e_{a_{k}})=(-1)^{k}I_{\gamma^{-1}}(e_{a_{k}}\cdots e_{a_{1}})$.
        \item(Shuffle product formula) $I_{\gamma}(w\sh w')=I_{\gamma}(w)I_{\gamma}(w')$.
    \end{enumerate} 
\end{proposition}
\subsection{Duality for complex MPLs}
For $X\in\bbC$, let $\dch_{0,X}$ (resp.~$\dch_{\infty,X}$) denote the path $t\mapsto Xt$ (resp.~$t\mapsto X/t$).
Especially we put $\dch\coloneqq\dch_{0,1}$.
Then $\fH_{\dch}$ is spanned over $\bbQ$ by $e_{a_{1}}\cdots e_{a_{k}}$ with $a_{1},\ldots,a_{k}\in\bbC\setminus(0,1)$.
The following proposition is immediately obtained by the expansion $(t-1/z_{i})^{-1}=-z_{i}(1+z_{i}t+z_{i}^{2}t^{2}+\cdots)$.
\begin{proposition}\label{prop:mpl_ii}
    For $v=\vv{z_{1},\ldots,z_{r}}{k_{1},\ldots,k_{r}}\in\cV$, we have
    \begin{equation}\label{eq:mpl_ii}
        \Li^{\sh}(v)=(-1)^{r}I_{\dch}(e_{1/z_{1}}e_{0}^{k_{1}-1}\cdots e_{1/z_{r}}e_{0}^{k_{r}-1}).
    \end{equation}
\end{proposition}
\begin{remark}
    Whereas the element $\vv{z_{1},\ldots,z_{r}}{k_{1},\ldots,k_{r}}$ in $\cV$ is restricted as $|z_{1}|,\ldots,|z_{r}|\le 1$ and $(k_{r},z_{r})\neq (1,1)$, the right-hand side of \eqref{eq:mpl_ii} is defined whenever $z_{i}\in\bbC\setminus\bbR_{>1}$.
    From here, we use the expression \eqref{eq:mpl_ii} to define $\Li^{\sh}$ generally for such elements. 
\end{remark}
We present a proof of Theorem \ref{thm:mpl_duality} via the iterated integrals.
\begin{proof}[Proof of Theorem \ref{thm:mpl_duality} given in \cite{bbbl01}]
    By Theorem \ref{prop:mpl_ii} we have
    \[\Li^{\sh}\vv{\bz}{\bk}=(-1)^{\dep(\bk)}I_{\dch}(W(\bl_{1})e_{1}^{a_{1}-1}e_{1/w_{1}}e_{0}^{b_{1}-1}\cdots W(\bl_{d-1})e_{1}^{a_{d-1}-1}e_{1/w_{d-1}}e_{0}^{b_{d-1}-1}W(\bl_{d})),\]
    where $W(l_{1},\ldots,l_{s})\coloneqq e_{1}e_{0}^{l_{1}-1}\cdots e_{1}e_{0}^{l_{s}-1}$.
    Substituting $t\mapsto 1-t$ and applying \eqref{eq:transformation_dual} and Proposition \ref{prop:basic_ii} (2), we obtain
    \begin{align}
        &\Li^{\sh}\vv{\bz}{\bk}\\
        &=(-1)^{\dep(\bk)}I_{\dch}(W((\bl_{d})^{\dagger})e_{1}^{b_{d-1}-1}e_{(w_{d-1}-1)/w_{d-1}}e_{0}^{a_{d-1}-1}\cdots W((\bl_{2})^{\dagger})e_{1}^{b_{1}-1}e_{(w_{1}-1)/w_{1}}e_{0}^{a_{1}-1}W((\bl_{1})^{\dagger}))\\
        &=\Li^{\sh}\left(\vv{\bz}{\bk}^{\dagger}\right).
    \end{align}
\end{proof}
\subsection{Duality for FMPLs}
The original proof of Theorem \ref{thm:fmpl_duality} in \cite{seki19} is executed by some algebraic arguments and applying the $p$-adic expansion of binomial coefficients in \emph{Sakugawa--Seki's identity} \cite[Theorem 2.5]{ss17}
\begin{multline}\label{eq:ss}
    \sum_{0=n_{0}<n_{1}\le\cdots\le n_{r}<N}\frac{z_{1}^{n_{1}-n_{0}}\cdots z_{r}^{n_{r}-n_{r-1}}}{n_{1}^{k_{1}}\cdots n_{r}^{k_{r}}}\cdot (-1)^{n_{r}}\binom{N-1}{n_{r}}\\
    =\Li_{<N}^{\sh,\star}\vv{1-z_{1},\{1\}^{k_{1}-1},\ldots,1-z_{r},\{1\}^{k_{r}-1}}{\{1\}^{\wt(\bk)}}-\Li_{<N}^{\sh,\star}\vv{\{1\}^{k_{1}},1-z_{2},\{1\}^{k_{2}-1},\ldots,1-z_{r},\{1\}^{k_{r}-1}}{\{1\}^{\wt(\bk)}},
\end{multline}
which holds an index $\bk=(k_{1},\ldots,k_{r})$ with $r\ge 1$ in $\bbQ[z_{1},\ldots,z_{r}]$.
In this subsection, we give a proof of \eqref{eq:ss} using iterated integrals.
\begin{lemma}\label{lem:star_mpl_ii}
    For $\vv{z_{1},\ldots,z_{r}}{k_{1},\ldots,k_{r}}\in\cV$, we have
    \[\Li^{\sh,\star}\vv{z_{1},\ldots,z_{r}}{k_{1},\ldots,k_{r}}=(-1)^{\wt(\bk)}I_{\dch_{\infty,1}}((e_{z_{1}}-e_{0})e_{0}^{k_{1}-1}e_{z_{2}}e_{0}^{k_{2}-1}\cdots e_{z_{r}}e_{0}^{k_{r}-1}).\]
\end{lemma}
\begin{proof}
    From \eqref{eq:star_nonstar} and Proposition \ref{prop:mpl_ii}, we have
    \begin{align}
        \Li^{\sh,\star}\vv{z_{1},\ldots,z_{r}}{k_{1},\ldots,k_{r}}
        &=\sum_{d=0}^{r-1}\sum_{0=i_{0}<\cdots<i_{d+1}=r}\Li^{\sh}\vv{z_{i_{0}+1},\ldots,z_{i_{d}+1}}{k_{i_{0}+1}+\ldots+k_{i_{1}},\ldots,k_{i_{d}+1}+\cdots+k_{i_{d+1}}}\\
        &=\sum_{d=0}^{r-1}\sum_{0=i_{0}<\cdots<i_{d+1}=r}(-1)^{d+1}I_{\dch}(e_{1/z_{i_{0}+1}}e_{0}^{k_{i_{0}+1}+\cdots+k_{i_{1}}}\cdots e_{1/z_{i_{d}+1}}e_{0}^{k_{i_{d}+1}+\cdots+k_{i_{d+1}}})\\
        &=\sum_{d=0}^{r-1}\sum_{0=i_{0}<\cdots<i_{d+1}=r}I_{\dch}(A_{1}e_{0}^{k_{1}-1}\cdots A_{r}e_{0}^{k_{r}-1}),
    \end{align}
    where
    \[A_{j}=\begin{cases}
        -e_{1/z_{j}} & \text{if }j=i_{c}+1\text{ for some }0\le c\le d,\\
        e_{0} & \text{otherwise.}
    \end{cases}\]
    Thus we have
    \[
        \Li^{\sh,\star}\vv{z_{1},\ldots,z_{r}}{k_{1},\ldots,k_{r}}=I_{\dch}((-e_{1/z_{1}})e_{0}^{k_{1}-1}(e_{0}-e_{1/z_{2}})e_{0}^{k_{2}-1}\cdots (e_{0}-e_{1/z_{r}})e_{0}^{k_{r}-1}).
    \]
    Then consider substitution $t\mapsto 1/t$: by \eqref{eq:transformation_inverse}, equalities $e_{1/z_{1}}(t^{-1})=(e_{z_{1}}-e_{0})(t)$, $e_{0}(t^{-1})=-e_{0}(t)$ and $(e_{0}-e_{1/z_{i}})(t^{-1})=-e_{z_{i}}(t)$ hold.
    Therefore we obtain the desired result.
\end{proof}
\begin{proof}[Proof of \eqref{eq:ss}]
    It suffices to prove \eqref{eq:ss} for complex numbers $z_{1},\ldots,z_{r}$ satisfying $|z_{i}|\le 1$ and $|1-z_{i}|\le 1$ ($1\le i\le r$).
    Consider the generating functions of both sides of \eqref{eq:ss}.
    For the left-hand side, we have
    \begin{align}
        &\sum_{N=1}^{\infty}\left(\sum_{0=n_{0}<n_{1}\le\cdots\le n_{r}<N}\frac{z_{1}^{n_{1}-n_{0}}\cdots z_{r}^{n_{r}-n_{r-1}}}{n_{1}^{k_{1}}\cdots n_{r}^{k_{r}}}\cdot (-1)^{n_{r}}\binom{N-1}{n_{r}}\right)X^{N}\\
        &=\sum_{\substack{0=n_{0}<n_{1}\le\cdots\le n_{r}\\ N\ge 0}}\frac{z_{1}^{n_{1}-n_{0}}\cdots z_{r}^{n_{r}-n_{r-1}}}{n_{1}^{k_{1}}\cdots n_{r}^{k_{r}}}\cdot (-1)^{n_{r}}\cdot\binom{N+n_{r}}{n_{r}}X^{N+n_{r}+1}\\
        &=\sum_{0=n_{0}<n_{1}\le\cdots\le n_{r}}\frac{z_{1}^{n_{1}-n_{0}}\cdots z_{r}^{n_{r}-n_{r-1}}}{n_{1}^{k_{1}}\cdots n_{r}^{k_{r}}}\cdot (-1)^{n_{r}}X(1-X)^{-n_{r}-1}\\
        &=\frac{X}{1-X}\sum_{0=n_{0}<n_{1}\le\cdots\le n_{r}}\prod_{i=1}^{r}\left(\frac{Xz_{i}}{X-1}\right)^{n_{i}-n_{i-1}}\frac{1}{n_{i}^{k_{i}}}\\
        &=\frac{X}{1-X}\Li^{\sh,\star}\vv{Xz_{1}/(X-1),\ldots,Xz_{r}/(X-1)}{k_{1},\ldots,k_{r}},
\end{align}
for $X\in\bbC$ with $0<\Re(X)<1/2$.
Since Lemma \ref{lem:star_mpl_ii} holds, substituting $t\mapsto (X-t)/(X-1)$, we have
\begin{align}
    &\sum_{N=1}^{\infty}\left(\sum_{0=n_{0}<n_{1}\le\cdots\le n_{r}<N}\frac{z_{1}^{n_{1}-n_{0}}\cdots z_{r}^{n_{r}-n_{r-1}}}{n_{1}^{k_{1}}\cdots n_{r}^{k_{r}}}\cdot (-1)^{n_{r}}\binom{N-1}{n_{r}}\right)X^{N}\\
    &=\frac{X}{1-X}(-1)^{\wt(\bk)}I_{\dch_{\infty,1}}((e_{Xz_{1}/(X-1)}-e_{0})e_{0}^{k_{1}-1}e_{Xz_{2}/(X-1)}e_{0}^{k_{2}-1}\cdots e_{Xz_{r}/(X-1)}e_{0}^{k_{r}-1})\\
    &=\frac{X}{1-X}(-1)^{\wt(\bk)}I_{\dch_{\infty,1}}((e_{X(1-z_{1})}-e_{X})e_{X}^{k_{1}-1}e_{X(1-z_{2})}e_{X}^{k_{2}-1}\cdots e_{X(1-z_{r})}e_{X}^{k_{r}-1})\\
    &=\begin{multlined}[t]
        \frac{X}{1-X}\left(\Li^{\sh,\star}\vv{X(1-z_{1}),\{X\}^{k_{1}-1},\ldots,X(1-z_{r}),\{X\}^{k_{r}-1}}{\{1\}^{\wt(\bk)}}\right.\\
        \left.-\Li^{\sh,\star}\vv{\{X\}^{k_{1}},X(1-z_{2}),\{X\}^{k_{2}-1},\ldots,X(1-z_{r}),\{X\}^{k_{r}-1}}{\{1\}^{\wt(\bk)}}\right)\end{multlined}\\
    &=\begin{multlined}[t]
        \sum_{N=1}^{\infty}\left(\Li^{\sh,\star}_{<N}\vv{1-z_{1},\{1\}^{k_{1}-1},\ldots,1-z_{r},\{1\}^{k_{r}-1}}{\{1\}^{\wt(\bk)}}\right.\\
        \left.-\Li^{\sh,\star}_{<N}\vv{\{1\}^{k_{1}},1-z_{2},\{1\}^{k_{2}-1},\ldots,1-z_{r},\{1\}^{k_{r}-1}}{\{1\}^{\wt(\bk)}}\right)X^{N}
    \end{multlined}
\end{align}
where we used $e_{X(1-z_{1})}-e_{X}=(e_{X(1-z_{1})}-e_{0})-(e_{X}-e_{0})$ in the third equality and Remark \ref{rem:generating_function} in the last equality.
\end{proof}
\section{Proof of the main theorem}\label{sec:main}
In this section, we fix an integer $\alpha$ and a finite subset $Q$ of $\bbC$ satisfying
\begin{enumerate}
    \item $\{0,1\}\subseteq Q$,
    \item $Q\cap(0,1)=\emp$,
    \item $\Re(u)>0$ and $|u-1/4|>1/4$ for all $u\in Q^{\times}$,
    \item $2\Re(u/v)<\Re(v)$ for all $u\in Q$ and $v\in Q^{\times}$ if $u\neq v$.
\end{enumerate}
Here $Q^{\times}$ denotes $Q\setminus\{0\}$. To such a $Q$, we associate the set $\cZ_{Q}\coloneqq I_{\dch}(\bbQ\langle e_{a/b}\mid a\in Q,~b\in Q^{\times}\rangle)\subseteq\bbR$ on which Proposition \ref{prop:basic_ii} (3) induces a $\bbQ$-algebra structure. Here and in what follows, we always regard $e_{\infty}$ as $0$.
Note that $\cZ_{\{0,1\}}$ is the $\bbQ$-algebra $\cZ$ appearing in \S\ref{sec:intro}.

Let $c$ be a path which circles $1$ counterclockwise satisfying
\[\lim_{t\to+0}\frac{c(t)-1}{t}=-1,\qquad\text{and}\qquad\lim_{t\to 1-0}\frac{c(t)-1}{t-1}=1.\]
Put $\beta\coloneqq\dch\cdot c\cdot\dch^{-1}$. 
\begin{definition}
    Let $z_{1},\ldots,z_{r}$ be elements of $Q$, and $\bk=(k_{1},\ldots,k_{r})$ an index.
    We define
    \begin{equation}\label{eq:rsmpl}
        \pounds_{\hcRS,\alpha}\vv{z_{1},\ldots,z_{r}}{k_{1},\ldots,k_{r}}\coloneqq\frac{(-1)^{r}}{2\pi i}\sum_{u\in Q^{\times}}u^{\alpha}\sum_{n=0}^{\infty}I_{\beta}(e_{u/z_{1}}e_{0}^{k_{1}-1}\cdots e_{u/z_{r}}e_{0}^{k_{r}-1}e_{u}e_{0}^{n})(-t)^{n}\in\bbC\jump{t},
    \end{equation}
    and
    \[\pounds_{\hcRS,\alpha}^{\star}\vv{z_{1},\ldots,z_{r}}{k_{1},\ldots,k_{r}}\coloneqq\sum_{d=0}^{r-1}\sum_{0=i_{0}<\cdots<i_{d+1}=r}\pounds_{\hcRS,\alpha}\vv{z_{i_{0}+1},\ldots,z_{i_{d}+1}}{k_{i_{0}+1}+\ldots+k_{i_{1}},\ldots,k_{i_{d}+1}+\cdots+k_{i_{d+1}}}.\]
\end{definition}
\begin{proposition}[{\cite[Proposition 4.17 (3)]{kawamura22}}]\label{prop:k}
    In the above settings, $\pounds_{\hcRS,\alpha}\vv{z_{1},\ldots,z_{r}}{k_{1},\ldots,k_{r}}$ belongs to $\cZ_{Q}[2\pi i]\jump{t}$ and coincides with $\pounds_{\hcS,\alpha}\vv{z_{1},\ldots,z_{r}}{k_{1},\ldots,k_{r}}$ modulo $2\pi i\cZ_{Q}[2\pi i]\jump{t}$.
\end{proposition}
Moreover, from an argument similar to \cite[p.~19]{kawamura22},
we see that
\begin{equation}\label{eq:conjugate}
    \frac{(-1)^{r}}{2\pi i}\sum_{u\in Q^{\times}}u^{\alpha}\sum_{n=0}^{\infty}I_{\beta^{-1}}(e_{u/z_{1}}e_{0}^{k_{1}-1}\cdots e_{u/z_{r}}e_{0}^{k_{r}-1}e_{u}e_{0}^{n})(-t)^{n}=\overline{\pounds_{\hcRS,\alpha}\vv{\overline{z_{1}},\ldots,\overline{z_{r}}}{k_{1},\ldots,k_{r}}}
\end{equation}
holds, where the longest overline on the right-hand side means the coefficientwise complex conjugate.
\begin{proposition}\label{prop:star_smpl_ii}
    Let $z_{1},\ldots,z_{r}$ be elements of $Q^{\times}$ and $(k_{1},\ldots,k_{r})$ be an index.
    For $u\in Q^{\times}$, let $\eta_{u}$ be the path $t\mapsto 1/2u+i(t^{-1}-(1-t)^{-1})$.
    Then we have
    \[\pounds_{\hcRS,\alpha}^{\star}\vv{z_{1},\ldots,z_{r}}{k_{1},\ldots,k_{r}}=\frac{(-1)^{\wt(\bk)}}{2\pi i}\sum_{u\in Q^{\times}}u^{\alpha}I_{\eta_{u}}((e_{z_{1}/u}-e_{0})e_{0}^{k_{1}-1}e_{z_{2}/u}e_{0}^{k_{2}-1}\cdots e_{z_{r}/u}e_{0}^{k_{r}-1}(f_{1/u,0,t}-f_{0,0,t})),\]
    where $f_{a,b,t}$ is the symbol identified with the $1$-form $(x-a)^{-1}\exp(t\log(x-b))\,dx$ and the branch of $\log(\eta_{u}(x)-u^{-1})$ is taken to be what satisfies
    \[\lim_{x\to 1}(\log(\eta_{u}(x)-u^{-1})-\log(\eta_{u}(x)))=0.\]
\end{proposition}
\begin{proof}
    For $0<z<z'<1$ and a non-negative integer $n$, we have
    \begin{align}
        \sum_{n=0}^{\infty}\frac{1}{\beta(z)-u}\int_{z<x_{1}<\cdots<x_{n}<z'}\frac{d\beta(x_{1})}{\beta(x_{1})}\cdots\frac{d\beta(x_{n})}{\beta(x_{n})}(-t)^{n}
        &=\frac{1}{\beta(z)-u}\sum_{n=0}^{\infty}\frac{1}{n!}\left(\int_{z}^{z'}\frac{d\beta(x)}{\beta(x)}\right)^{n}(-t)^{n}\\
        &=\frac{1}{\beta(z)-u}\exp(t\log\beta(z')-t\log\beta(z)).
    \end{align}
    Its asymptotic expansion as $z'\to 1-0$ take the form
       \[\frac{1}{\beta(z)-u}\exp(t\log\beta(z')-t\log\beta(z))=\frac{1}{\beta(z)-u}\exp(-t\log\beta(z))+\sum_{i=1}^{\mu}a_{i}(\log (1-z'))^{i}+O((1-z')\log(1-z')^{J})\]
        for some $\mu\in\bbZ_{\ge 0}$, $a_{1},\ldots,a_{\mu}\in\bbC$ and $J>0$.
    Thus we obtain
    \begin{align}
        \pounds_{\hcRS,\alpha}\vv{z_{1},\ldots,z_{r}}{k_{1},\ldots,k_{r}}
        &=\frac{(-1)^{r}}{2\pi i}\sum_{u\in Q^{\times}}u^{\alpha}I_{\beta}(e_{u/z_{1}}e_{0}^{k_{1}-1}\cdots e_{u/z_{r}}e_{0}^{k_{r}-1}f_{u,0,-t}).
    \end{align}
    Similarly to the proof of Lemma \ref{lem:star_mpl_ii}, we have
    \begin{align}
        \pounds_{\hcRS,\alpha}^{\star}\vv{z_{1},\ldots,z_{r}}{k_{1},\ldots,k_{r}}
        &=\frac{(-1)^{\wt(\bk)}}{2\pi i}\sum_{u\in Q^{\times}}u^{\alpha}I_{\gamma}((e_{z_{1}/u}-e_{0})e_{0}^{k_{1}-1}\cdots e_{z_{r}/u}e_{0}^{k_{r}-1}(f_{1/u,0,t}-f_{0,0,t})),
    \end{align}
    where $\gamma$ is the path that first runs on the real axis from $\infty$ to $1$, then counterclockwise around $1$, and then back on the real axis from $1$ to $\infty$.
    Since we can deform $\gamma$ to $\eta_{u}$ continuously on $\bbP^{1}(\bbC)\setminus\{a/b\mid a\in Q,~b\in Q^{\times}\}$, we get the desired assertion.
\end{proof}
\begin{remark}
    By the definition of $\pounds_{\hcRS,\alpha}^{\star}$, the equality
    \[\pounds_{\hcRS,\alpha}^{\star}\vv{z_{1},\ldots,z_{r},0,w_{1},\ldots,w_{s}}{k_{1},\ldots,k_{r},a,l_{1},\ldots,l_{s}}=\pounds_{\hcRS,\alpha}^{\star}\vv{z_{1},\ldots,z_{r},w_{1},\ldots,w_{s}}{k_{1},\ldots,k_{r-1},k_{r}+a,l_{1},\ldots,l_{s}}\]
    holds.
    Thus Proposition \ref{prop:star_smpl_ii} actually holds even if $0\in\{z_{1},\ldots,z_{r}\}$.
\end{remark}
\begin{corollary}\label{cor:star_smpl_ii}
    For an index $(k_{1},\ldots,k_{r})$ and $z_{1},\ldots,z_{r}\in Q$, we have
    \[\sum_{n=0}^{\infty}\pounds_{\hcRS,\alpha}^{\star}\vv{z_{1},\ldots,z_{r},\{1\}^{n}}{k_{1},\ldots,k_{r},\{1\}^{n}}t^{n}=\frac{(-1)^{\wt(\bk)}}{2\pi i}\sum_{u\in Q^{\times}}u^{\alpha}I_{\eta_{u}}((e_{z_{1}/u}-e_{0})e_{0}^{k_{1}-1}\cdots e_{z_{r}/u}e_{0}^{k_{r}-1}(f_{1/u,1/u,t}-f_{0,0,t})).\]
\end{corollary}
\begin{proof}
    Take $0<z<z'<1$ and $u\in Q^{\times}$.
    Then we have
    \begin{align}
        \sum_{n=0}^{\infty}\int_{z<x_{1}<\cdots<x_{n}<z'}\frac{d\eta_{u}(x_{1})}{\eta_{u}(x_{1})-u^{-1}}\cdots\frac{d\eta_{u}(x_{n})}{\eta_{u}(x_{n})-u^{-1}}t^{n}
        &=\sum_{n=0}^{\infty}\frac{1}{n!}\left(\int_{z}^{z'}\frac{d\eta_{u}(x)}{\eta_{u}(x)-u^{-1}}\right)^{n}t^{n}\\
        &=\exp(t\log(\eta_{u}(z')-u^{-1})-t\log(\eta_{u}(z)-u^{-1})),
    \end{align}
    and integrating it by $z'$ we obtain
    \begin{align}
        &\sum_{n=0}^{\infty}\int_{z<z'<1}\left(\int_{z<x_{1}<\cdots<x_{n}<z'}\frac{d\eta_{u}(x_{1})}{\eta_{u}(x_{1})-u^{-1}}\cdots\frac{d\eta_{u}(x_{n})}{\eta_{u}(x_{n})-u^{-1}}\right)(f_{1/u,0,t}-f_{0,0,t})(\eta_{u}(z'))t^{n}\\
        &=\int_{z<z'<1}\left(\frac{1}{\eta(z')-u^{-1}}-\frac{1}{\eta(z')}\right)\exp\left(t\log(\eta_{u}(z')-u^{-1})+t\log\eta_{u}(z')-t\log(\eta_{u}(z)-u^{-1})\right)\,d\eta_{u}(z')\\
        &=-\exp(t\log(\eta_{u}(z)-u^{-1}))\int_{z<z'<1}\left(\frac{d}{d\eta_{u}(x)}\frac{1}{t}\exp(t\log\eta(z')-t\log(\eta(z')-u^{-1}))\right)\,d\eta_{u}(z')\\
        &=\exp(t\log(\eta_{u}(z)-u^{-1}))\cdot\frac{1}{t}\left(\exp(t\log\eta(z)-t\log(\eta(z)-u^{-1}))-1\right)\\
        &=\frac{1}{t}\left(\exp(t\log\eta(z))-\exp(t\log(\eta(z)-u^{-1}))\right)\\
        &=\int_{z<z'<1}(f_{1/u,1/u,t}-f_{0,0,t})(z'),
    \end{align}
    where we formally defined $(f-1)/t$ for $f\in 1+t\bbC\jump{t}$.
    By combining this and Proposition \ref{prop:star_smpl_ii}, the desired formula is shown.
\end{proof}
Define
    \[\cL_{\hcRS,\alpha}^{\star}\vv{z_{1},\ldots,z_{r}}{k_{1},\ldots,k_{r}}\coloneqq e^{\pi it/2}\sum_{n=0}^{\infty}\left(\pounds_{\hcRS,\alpha}^{\star}\vv{z_{1},\ldots,z_{r},\{1\}^{n}}{k_{1},\ldots,k_{r},\{1\}^{n}}-\frac{1}{2}\pounds_{\hcRS,\alpha}^{\star}\vv{1,z_{2},\ldots,z_{r},\{1\}^{n}}{k_{1},\ldots,k_{r},\{1\}^{n}}\right)t^{n}\]
    for $z_{1},\ldots,z_{r}\in Q$ and an index $(k_{1},\ldots,k_{r})\neq\emp$.
The following is the main theorem of this paper.
\begin{theorem}\label{thm:rsmpl_duality}
    Let $z_{1},\ldots,z_{d}$ ($d\ge 1$) be elements of $Q$ and $\bk_{1},\ldots,\bk_{d}$ indices.
    Then we have
    \[
        \cL^{\star}_{\hcRS,\alpha}\vv{z_{1},\{1\}^{\dep(\bk_{1})-1},\ldots,z_{d},\{1\}^{\dep(\bk_{d})-1}}{\bk_{1},\ldots,\bk_{d}}
        =\overline{\cL^{\star}_{\hcRS,\alpha}\vv{1-\overline{z_{1}},\{1\}^{\dep((\bk_{1})^{\vee})-1},\ldots,1-\overline{z_{d}},\{1\}^{\dep((\bk_{d})^{\vee})-1}}{(\bk_{1})^{\vee},\ldots,(\bk_{d})^{\vee}}}.
    \]
\end{theorem}
\begin{proof}
    Put $k\coloneqq\wt(\bk_{1})+\cdots+\wt(\bk_{d})$.
    Use the notation $\omega_{u}(l_{1},\ldots,l_{s})\coloneqq e_{0}^{l_{1}-1}e_{1/u}\cdots e_{0}^{l_{s-1}-1}e_{1/u}e_{0}^{l_{s}-1}$ for $u\in Q^{\times}$ and an index $(l_{1},\ldots,l_{s})\neq\emp$.
    From Corollary \ref{cor:star_smpl_ii} we obtain
    \begin{multline}
        e^{-\pi it/2}\cL_{\hcRS,\alpha}^{\star}\vv{w_{1},\ldots,w_{r}}{k_{1},\ldots,k_{r}}\\
        =\frac{(-1)^{k}}{2\pi i}\sum_{u\in Q^{\times}}u^{\alpha}I_{\eta_{u}}((e_{w_{1}/u}-(e_{1/u}+e_{0})/2)e_{0}^{k_{1}-1}e_{w_{2}/u}e_{0}^{k_{2}-1}\cdots e_{w_{r}/u}e_{0}^{k_{r}-1}(f_{1/u,1/u,t}-f_{0,0,t}))
    \end{multline}
    for $w_{1},\ldots,w_{r}\in Q$ and $k_{1},\ldots,k_{r}\in\bbZ_{\ge 1}$.
    Therefore it follows that
    \begin{align}
        &e^{-\pi it/2}\cL_{\hcRS,\alpha}^{\star}\vv{z_{1},\{1\}^{\dep(\bk_{1})-1},\ldots,z_{r},\{1\}^{\dep(\bk_{d})-1}}{\bk_{1},\ldots,\bk_{d}}\\
        &=\frac{(-1)^{k}}{2\pi i}\sum_{u\in Q^{\times}}u^{\alpha}I_{\eta_{u}}((e_{z_{1}/u}-(e_{1/u}+e_{0})/2)\omega_{u}(\bk_{1})e_{z_{2}/u}\omega_{u}(\bk_{2})\cdots e_{z_{d}/u}\omega_{u}(\bk_{d})(f_{1/u,1/u,t}-f_{0,0,t})).
    \end{align}
    Substituting $z\mapsto 1/u-z$ and using \eqref{eq:transformation_dual}, we obtain
    \begin{align}
        &e^{-\pi it/2}\cL_{\hcRS,\alpha}^{\star}\vv{z_{1},\{1\}^{\dep(\bk_{1})-1},\ldots,z_{r},\{1\}^{\dep(\bk_{d})-1}}{\bk_{1},\ldots,\bk_{d}}\\
        &=\begin{multlined}[t]
            \frac{(-1)^{k}}{2\pi i}e^{-\pi it}\sum_{u\in Q^{\times}}u^{\alpha}\\
        \cdot I_{\eta_{u}^{-1}}((e_{(1-z_{1})/u}-(e_{0}+e_{1/u})/2)\omega_{u}((\bk_{1})^{\vee})e_{(1-z_{2})/u}\omega_{u}((\bk_{2})^{\vee})\cdots e_{(1-z_{d})/u}\omega_{u}((\bk_{d})^{\vee})(-f_{0,0,t}+f_{1/u,1/u,t}))
        \end{multlined}\\
        &=e^{-\pi it/2}\tilde{\cL}_{\hcRS,\alpha}^{\star}\vv{1-z_{1},\{1\}^{\dep((\bk_{1})^{\vee})-1},\ldots,1-z_{r},\{1\}^{\dep((\bk_{d})^{\vee})-1}}{(\bk_{1})^{\vee},\ldots,(\bk_{d})^{\vee}},
    \end{align}
    where $\tilde{\cL}$ is defined similarly as $\cL$ except for that $\eta_{u}$ (resp.~$e^{\pi it/2}$) is replaced by $\eta_{u}^{-1}$ (resp.~$e^{-\pi it/2}$).
    Thus we get the theorem by this equation and \eqref{eq:conjugate}.
\end{proof}
Theorem \ref{thm:main} follows from Theorem \ref{thm:rsmpl_duality} and Proposition \ref{prop:k}.

\end{document}